\documentclass[12pt]{amsart}

\usepackage{amscd}
\usepackage{a4wide}
\usepackage{amssymb}
\usepackage{amstext}
\usepackage{amsthm}
\usepackage{mathrsfs}

\usepackage{latexsym}
\usepackage{xcolor}
\numberwithin{equation}{section}

\DeclareMathOperator{\Hol}{Hol}

\newcommand{\DD}{{\mathbb D}}
\newcommand{\CC}{{\mathbb C}}

\renewcommand{\phi}{\varphi}

\newcommand{\ve}{{\varepsilon}}

\newtheorem{Thm}{Theorem}%[section]
\newtheorem{theorem}[Thm]{Theorem}
\newtheorem{lemma}[Thm]{Lemma}

\newtheorem{remark}[Thm]{Remark}

\begin{document}
\sloppy
\title[The Young type theorem in weighted Fock spaces]
{The Young type theorem in weighted Fock spaces}

\author{Anton Baranov, Yurii Belov, Alexander Borichev}
\address{Anton Baranov:
\newline Department of Mathematics and Mechanics, St.~Petersburg State University, St.~Petersburg, Russia
\newline {\tt anton.d.baranov@gmail.com}
\smallskip
\newline \phantom{x}\,\, Yurii Belov:
\newline St.~Petersburg State University, St. Petersburg, Russia
\newline {\tt j\_b\_juri\_belov@mail.ru}
\smallskip
\newline \phantom{x}\,\, Alexander Borichev:
\newline Aix-Marseille Universit\'e, CNRS, Centrale Marseille, I2M, Marseille, France
\newline {\tt alexander.borichev@math.cnrs.fr}
}

\thanks{The work was supported by Russian Science Foundation grant 14-41-00010.}



\begin{abstract} 
We prove that for every radial weighted Fock space, the system biorthogonal 
to a complete and minimal system of reproducing kernels is also
complete under very mild regularity assumptions on the weight.
This result generalizes a theorem by Young on reproducing kernels 
in the Paley--Wiener space and a recent result of Belov 
for the classical Bargmann--Segal--Fock space.
\end{abstract}

\maketitle

\section{Introduction and the main result}

Let $\mathcal{H}$ be a reproducing kernel 
Hilbert space whose elements are entire functions. 
This means that the evaluation functionals $f\mapsto f(\lambda)$
are continuous on $\mathcal{H}$ for every $\lambda\in \mathbb{C}$, 
and so there exist functions $k_\lambda \in \mathcal{H}$
({\it reproducing kernels}) such that
$$
f(\lambda) = \langle f, k_\lambda\rangle_{\mathcal{H}}, \qquad f\in \mathcal{H}.
$$
It is well known that
the structure properties of $\mathcal{H}$ are closely related to 
geometric properties of systems of reproducing kernels.
E.g., one of the simplest relations of this type 
says that a set $\Lambda = \{\lambda_n\}$ 
is a uniqueness set for $\mathcal{H}$ if and only if 
the system of reproducing kernels $\{k_\lambda\}_{\lambda \in \Lambda}$
is complete in $\mathcal{H}$.

We say that the space $\mathcal{H}$ has the {\it division property}
if  for any $f\in \mathcal{H}$ and any $w\in \mathbb{C}$ such that
$f(w) = 0$,  we have $\frac{f(z)}{z-w} \in \mathcal{H}$.
Most of the classical spaces of entire functions
(Paley--Wiener, de Branges, Fock spaces) possess this property.  

Now assume that a system $\{k_\lambda\}_{\lambda\in \Lambda}$ is complete
and {\it minimal} in $\mathcal{H}$ (i.e., $k_\lambda$ 
does not belong to the closed linear span 
of $\{k_\mu\}_{\mu\in \Lambda\setminus\{\lambda\}}$, $\lambda\in \Lambda$).
Fix $\lambda_0 \in\Lambda$ and let $g$ be a function in $\mathcal{H}$
which is orthogonal to $\{k_\mu\}_{\mu\in \Lambda\setminus\{\lambda_0\}}$
and so vanishes on $\Lambda\setminus\{\lambda_0\}$.
Then the function $G=(\cdot-\lambda_0) g$ (known as the 
{\it generating function} of the sequence 
$\{k_\lambda\}_{\lambda\in \Lambda}$ and unique up to a multiplicative constant) vanishes on $\Lambda$
and, by the division property, 
$\frac{G(z)}{z-\lambda} \in \mathcal{H}$ for any $\lambda\in \Lambda$.
Clearly, the functions 
\begin{equation}
\label{bior}
g_\lambda(z) = \frac{G(z)}{G'(\lambda) (z-\lambda)}
\end{equation}
form a system biorthogonal to $\{k_\lambda\}_{\lambda\in \Lambda}$,
i.e., $\langle k_\lambda, g_\mu\rangle = \delta_{\lambda \mu}$, $\lambda,\mu\in \Lambda$.

Given any function $f\in \mathcal{H}$, consider its (formal) Fourier
series with respect to the biorthogonal pair $\{k_\lambda\}_{\lambda\in \Lambda}$,
$\{g_\lambda\}_{\lambda\in \Lambda}$:
\begin{equation}
\label{lagr}
f\ \ \sim \ \ \sum_{\lambda \in \Lambda} \langle f, k_\lambda\rangle g_\lambda = 
     \sum_{\lambda \in \Lambda} f(\lambda)\frac{G}{G'(\lambda) (\cdot-\lambda)},
\end{equation}
which is nothing else than the Lagrange-type interpolation series for $f$.
The possibility of reconstruction of a function $f$ (in a given class)
from its Lagrange series \eqref{lagr}  is one of the 
basic problems of function theory. In this note we address one special 
instance of this problem, namely the question of completeness 
of the biorthogonal system $\{g_\lambda\}_{\lambda\in \Lambda}$.
Clearly, this is a necessary (though by no means sufficient)
condition for any reasonable reconstruction method. More delicate necessary 
conditions were studied in \cite{BBB1, BBB2}.

It is well known and easy to see that in general a system 
biorthogonal to a complete and minimal system in a Hilbert space 
need not to be complete. For special systems of functions 
the situation may be different. R.~Young \cite{young} showed that every complete 
and minimal system of exponentials $\{e^{i\lambda t}\}$ in $L^2(-\pi, \pi)$
(or, equivalently, any complete 
and minimal system of reproducing kernels in the Paley--Wiener space 
$PW_\pi$) has a complete biorthogonal system. This result was also 
obtained independently and by different methods 
by G.~Gubreev and A.~Kovalenko \cite{gub}. In \cite{bb11} 
the first two authors
studied completeness of biorthogonal systems in the context of 
de Branges spaces (or model subspaces of the Hardy space)
and showed that in these spaces there exist 
complete and minimal systems of reproducing kernels whose biorthogonal
system is incomplete with any given (even infinite) defect.

Recently, Yu.~Belov proved a Young-type result for the classical 
Bargmann--Segal--Fock space $\mathcal{F}$, see the definition below. 
The aim of this note is to extend this result to all radial 
weighted Fock spaces  satisfying very mild restrictions on the weight.
In particular, our method provides a simple proof for the classical Fock 
space $\mathcal{F}$. This method can also be easily modified to work in other
reproducing kernel Hilbert spaces of entire functions
(e.g., we show how to deduce the result of Young using this method).

Now we introduce the radial weighted Fock spaces. Let 
$h:\mathbb R_+\to\mathbb R_+$ be a continuous function.  
%such that $\log t = o(h(t))$, 
%$t\to\infty$. 

Extend $h$ to $\mathbb{C}$ by $h(z):= h(|z|)$, $z\in\mathbb{C}$. Consider the class of functions 
$$
\mathcal F_h=\Big\{f\in\Hol(\CC):\|f\|^2_{\mathcal F_h}=\int_{\CC}|f(z)|^2e^{-h(z)}\,dm_2(z)<\infty\Big\},
$$ 
$dm_2$ being planar Lebesgue measure. Each space $\mathcal F_h$ is a reproducing kernel 
Hilbert space with respect to the norm $\|\cdot\|_{\mathcal F_h}$. The classical Fock 
(or Bargmann--Segal--Fock) space corresponds to $h(r)= \pi r^2$.

Our main result is the following theorem: 

\begin{theorem} \label{th1} 
Let $h: \mathbb R_+\to\mathbb R_+$ be a function 
in $C^2(\mathbb R_+)$ such that 
\begin{equation}
\log(t+|h'(t)|+|h''(t)|)=o(h(t)),\qquad t\to\infty. \label{sta}
\end{equation}
Let $\{k_\lambda\}_{\lambda\in\Lambda}$ be a complete and minimal system 
of reproducing kernels in $\mathcal F_h$. Then the system
$\{g_\lambda\}_{\lambda\in\Lambda}$ biorthogonal to the system 
$\{k_\lambda\}_{\lambda\in\Lambda}$ is complete in $\mathcal F_h$. 
\label{mt}
\end{theorem}

We do not consider the case of finite dimensional Fock spaces $\mathcal F_h$, $\log(t)\not=o(h(t))$, 
$t\to\infty$.
\bigskip

%%%%%%%%%%%%%%%%%%%%%%%%%%%%%%%%%%%%%%%%%%%%%%

\section{Preliminaries}

In this section we establish some standard growth estimates for the elements
of the space $\mathcal F_h$. 
%Without loss of generality we will assume that
%$\log(t)=o(h(t))$, otherwise $\mathcal F_h$ is finite-dimensional and 
%Theorem \ref{mt} becomes trivial. 
In what follows we denote by $D(z, r)$ 
the open disc with center $z$ of radius $r$
and we write $\mathbb{D}$ instead of $D(0,1)$.

\begin{lemma} Given $\ve>0$, we have
$$
|h(w)-h(z)|=o(1),\qquad w\in D(z,e^{-\ve h(z)}),\, |z|\to\infty.
$$
\end{lemma} 

\begin{proof} Fix $\delta>0$. For large $t>0$ denote by $s$ a point in the interval 
$[t-\exp(-\ve h(t)),t+\exp(-\ve h(t))]$ closest to $t$ and such that $|h(s)-h(t)|\ge \delta$ 
(if such points exist at all; otherwise, there is nothing to prove).

Then by \eqref{sta} for some point $s'$ between $s$ and $t$ we have 
$$
\delta\le |h(s)-h(t)|\le |s-t|\cdot |h'(s')|\le C \exp(-\ve h(t)) \exp((\ve/2) h(t)),
$$
which is impossible for large $t$.
\end{proof}

\begin{lemma} 
\label{le3}
Given $\ve>0$ there exists $C_\ve>0$ such that 
$$
|f(z)|+|f'(z)|\le C_\ve e^{((1/2)+\ve)h(z)}\|f\|_{\mathcal F_h},\qquad f\in \mathcal F_h, \ z\in\CC.
$$
\end{lemma} 

\begin{proof} (Compare to the proof of Lemma~4.1 in \cite{jfa}.) 
Given $z\in\CC$, $\beta>0$, set $\gamma=\exp(-\beta h(z))$, 
$$
\psi(w)=h(z+\gamma w),\qquad w\in\DD.
$$
Then by \eqref{sta},
\begin{multline*}
\|\nabla \psi(w)\|+|\Delta \psi(w)|
\\ \lesssim \gamma|h'(|z+\gamma w|)|+\gamma^2\Bigl(|h''(|z+\gamma w|)|+
\frac{|h'(|z+\gamma w|)|}{|z+\gamma w|}\Bigr)=o(1)
\end{multline*}
uniformly in $w\in \DD$ when $|z|\to\infty$. Here and later on we write $a\lesssim b$ if $a\le Cb$ for some constant $C$. Set 
$$
\rho(z)=\int_{\DD}\log\Bigl|\frac{z-w}{1-\bar wz}\Bigr|\Delta \psi(w)\,dm_2(w),\qquad z\in \DD.
$$
Then
$$
|\rho(w)|+\|\nabla \rho(w)\|=o(1),\qquad 
w\in \DD,\,|z|\to\infty,
$$
and $\psi_1=\psi-\rho$ is real and harmonic on $\DD$. Denote by $\widetilde{\psi_1}$ the harmonic conjugate of $\psi_1$ and set 
$$
F(w)=f(z+\gamma w)e^{-(\psi_1(w)+i\widetilde{\psi_1}(w))/2},\qquad 
w\in \DD.
$$
Then $F$ is analytic on $\DD$ and 
$$
\int_\DD |F(w)|^2\,dm_2(w)\lesssim \gamma^{-2}\int_{D(z,\gamma)}|f(\zeta)|^2e^{-h(\zeta)}\,dm_2(\zeta)\lesssim \gamma^{-2}\|f\|^2_{\mathcal F_h}.
$$
Therefore,
\begin{align*}
|f(z)|&\lesssim |F(0)|e^{h(z)/2}\lesssim \gamma^{-1}e^{h(z)/2}\|f\|_{\mathcal F_h},\\
|f'(z)|&\lesssim  \gamma^{-1}|F'(0)|e^{h(z)/2}+ \gamma^{-1}
\|\nabla\psi_1(0)\||F(0)|e^{h(z)/2}\\ 
&\qquad\qquad\qquad\qquad\qquad\qquad\qquad\lesssim \gamma^{-2}e^{h(z)/2}\|f\|_{\mathcal F_h}.
\end{align*}
\end{proof}

%%%%%%%%%%%%%%%%%%%%%%%%%%%%%%%%%%%%%%%%%%%%%%%

\section{Proof of Theorem~\ref{th1}}

Fix a complete and minimal system $\{k_\lambda\}_{\lambda\in\Lambda}$ in $\mathcal F_h$. 
Its biorthogonal system
is given by \eqref{bior}, where the generating function $G$ 
has simple zeros exactly at $\Lambda$. Suppose that the system
$\{g_\lambda\}_{\lambda\in\Lambda}$ is not complete and 
fix $H\in \mathcal F_h\setminus\{0\}$ orthogonal 
to this system. Since polynomials are dense in any radial Fock space, we can
choose $n\ge 0$ such that 
$\langle H,z^n\rangle\not=0$ and fix a polynomial $P$ of degree $n+1$ 
with simple zeros $w_1,\ldots,w_{n+1}$ in $\CC\setminus\Lambda$. 

Given $z\in\CC$, the function 
$$
\zeta\mapsto \frac{G(z)P(\zeta)-G(\zeta)P(z)}{z-\zeta}
$$
is in $\mathcal F_h$. Set
$$
Q(z)=\int_\CC  \frac{G(z)P(\zeta)-G(\zeta)P(z)}{z-\zeta} \,\overline{H(\zeta)}\,e^{-h(\zeta)}\,dm_2(\zeta),\qquad z\in\CC.
$$
Then $Q$ is an entire function vanishing at $\Lambda$, and hence, 
$$
Q=GT
$$
for an entire function $T$. For every zero $w_j$ of $P$ we have 
$$
Q(w_j)=\int_\CC  \frac{G(w_j)P(\zeta)}{w_j-\zeta} \,\overline{H(\zeta)}\,e^{-h(\zeta)}\,dm_2(\zeta).
$$
Therefore, if $T\equiv 0$, then $H$ is orthogonal to all $P/(w_j-\cdot)$, and hence to 
$z^n$, which is impossible. 
To complete the proof it remains to verify that $Q\in \mathcal F_h$. 

We have 
\begin{multline}
\label{qqq}
Q(z)=G(z)\int_\CC  \frac{P(\zeta)\overline{H(\zeta)}}{z-\zeta} 
\,e^{-h(\zeta)}\,dm_2(\zeta) \\ -
P(z)\int_\CC  \frac{G(\zeta)\overline{H(\zeta)}}{z-\zeta}\,e^{-h(\zeta)}\,dm_2(\zeta)\\
\stackrel{\rm def}{=}
A_1(z)+A_2(z).
\end{multline}
The inclusion $Q\in \mathcal F_h$ will be a consequence 
of the following two statements.

\begin{lemma} \label{le5} Let $\phi\in L^1(\CC)\cap C^1(\CC)$, $\|\phi\|_{L^1(\CC)}\lesssim 1$. 
Suppose that 
$$
|\phi(z)|+\|\nabla \phi (z)\|\lesssim (1+|z|)^{-3},\qquad z\in\CC.  
$$
Then
$$
\Bigl| \int_\CC \frac{\phi(\zeta)\,dm_2(\zeta)}{z-\zeta}\Bigr|
\lesssim (1+|z|)^{-1},\qquad z\in\CC.
$$
\end{lemma} 

\begin{lemma} \label{le4} Let $\alpha>0$ and suppose that $\psi\in L^1(\CC)\cap C^1(\CC)$, 
$\|\psi\|_{L^1(\CC)}\lesssim 1$, 
$\|\nabla \psi\|\lesssim e^{\alpha h}$ on $\mathbb C$. Then  
$$
\Bigl| \int_\CC \frac{\psi(\zeta)\,dm_2(\zeta)}{z-\zeta}\Bigr|\lesssim e^{\alpha h(z)},\qquad z\in\CC.
$$
\end{lemma} 

\begin{proof}[Proof of Lemma~\ref{le5}] Let $z\in\CC$, $|z|\ge 2$. We have
\begin{gather*}
\Bigl| \int_\CC \frac{\phi(\zeta)\,dm_2(\zeta)}{z-\zeta}\Bigr|\le \Bigl| \int_{D(0,|z|/2)} \frac{\phi(\zeta)\,dm_2(\zeta)}{z-\zeta}\Bigr|
+\Bigl| \int_{D(z,1)} \frac{\phi(\zeta)\,dm_2(\zeta)}{z-\zeta}\Bigr| \\
+\Bigl| \int_{\CC\setminus(D(0,|z|/2)\cup D(z,1))} \frac{\phi(\zeta)\,dm_2(\zeta)}{z-\zeta}\Bigr|\\ \lesssim 
\frac{1}{1+|z|}+\int_{D(z,1)} \Bigl| \frac{\phi(\zeta)-\phi(z)}{z-\zeta}\Bigr| \,dm_2(\zeta)
+ \int_{|z|/2} \frac{dt}{t^2}
\lesssim \frac{1}{1+|z|}.
\end{gather*}
\end{proof}

\begin{proof}[Proof of Lemma~\ref{le4}] Let $z\in\CC$. Set $\eta=e^{-\alpha h(z)}$. We have
\begin{multline*}
\Bigl| \int_\CC \frac{\psi(\zeta)\,dm_2(\zeta)}{z-\zeta}\Bigr|\\ 
\le \Bigl| \int_{\CC\setminus D(z,\eta)} 
\frac{\psi(\zeta)\,dm_2(\zeta)}{z-\zeta}\Bigr|
+\Bigl| \int_{D(z,\eta)} \frac{\psi(\zeta)-\psi(z)}{z-\zeta}\,dm_2(\zeta)
\Bigr|\\ 
\lesssim e^{\alpha h(z)}+1\lesssim e^{\alpha h(z)}.
\end{multline*}
Here we use the fact that $\int_{D(z,r)} (z-\zeta)^{-1} dm_2(\zeta) =0$
for any disc $D(z, r)$.
\end{proof}

To complete the proof of Theorem~\ref{th1}, we apply Lemmas~\ref{le5} and \ref{le4} 
and estimate $A_1$ and $A_2$ in the right hand side of \eqref{qqq}. By Lemma~\ref{le3}, 
the function $\phi=P\overline{H}e^{-h}$ satisfies the conditions of Lemma~\ref{le5}, and hence 
$$
\Big| G(z)\int_\CC  \frac{P(\zeta)\overline{H(\zeta)}}{z-\zeta} 
\,e^{-h(\zeta)}\,dm_2(\zeta)\Big| \lesssim \frac{|G(z)|}{1+|z|}.
$$
Since $G/(1+|\cdot|) \in L^2(e^{-h}dm_2)$, we obtain that $A_1\in L^2(e^{-h}dm_2)$.

Next, choose $\lambda\in \Lambda$ and set $\psi=G\overline{H}e^{-h}/(\cdot-\lambda)$.  
By Lemma \ref{le3}, $\|\nabla \psi\|\lesssim e^{\alpha h}$ on $\mathbb C$ for every $\alpha>0$. 
Furthermore, $\psi\in L^1(\mathbb C)$. 
Therefore, by Lemma~\ref{le4}, 
$$
\Big|\int_\CC\frac{\psi(z)}{z-\zeta}\,dm_2(\zeta)\Big| \lesssim e^{\alpha h(z)},\qquad z\in\mathbb C,
$$
for every $\alpha>0$. 
Next, 
\begin{multline*}
\int_\CC  \frac{G(\zeta)\overline{H(\zeta)}}{z-\zeta}
\,e^{-h(\zeta)}\,dm_2(\zeta)  \\= 
(z-\lambda) \int_\CC  \frac{G(\zeta)\overline{H(\zeta)}}
{(\zeta-\lambda)(z-\zeta)} 
\,e^{-h(\zeta)}\,dm_2(\zeta)
-\int_\CC  \frac{G(\zeta)\overline{H(\zeta)}}{\zeta-\lambda} 
\,e^{-h(\zeta)}\,dm_2(\zeta) 
\\=(z-\lambda) \int_\CC  \frac{\psi %e^{-h}
}{z-\cdot} 
\,dm_2-\int_\CC  \psi\,dm_2.
\end{multline*}
%and $G(z)\overline{H(z)} (\lambda - z)^{-1} e^{-h(z)} \in L^1(\mathbb{C})$. 
Thus, %by Lemma \ref{le4}
$$
A_2=P\int_\CC \frac{G(\zeta)\overline{H(\zeta)}}{\cdot-\zeta}\,e^{-h(\zeta)}\,dm_2(\zeta) \in 
L^2(e^{-h}dm_2). 
$$
Finally, $Q\in \mathcal F_h$, which proves Theorem~\ref{th1}.\hfill \qed

\begin{remark}
{\rm Let us show how to prove Young's theorem on reproducing kernels
in the Paley--Wiener space $PW_\pi$ 
(the Fourier image of $L^2(-\pi, \pi)$) using a similar idea. Note that $PW_\pi$
does not contain polynomials. Instead we use the fact that
the family $\big\{\frac{\sin\pi z}{\pi (z-n)}\big\}_{n\in\mathbb{Z}}$ 
is an orthonormal basis in $PW_\pi$. Set
$$
S(z) = \frac{\sin\pi z}{z(z^2-1)}, \qquad \mathcal{Z} = \mathbb{Z} 
\setminus \{-1, 0,1\}.
$$
It is not too difficult to verify that the system 
$\big\{\frac{S(z)}{z-\mu}\big\}_{\mu\in\mathcal{Z}}$ is complete in $PW_\pi$.
Now assume that $\{k_\lambda\}_{\lambda\in\Lambda}$ is a complete and 
minimal system of reproducing kernels in $PW_\pi$ 
and $\{g_\lambda\}_{\lambda\in\Lambda}$ is its biorthogonal system
given by \eqref{bior}. Without loss of generality we may assume that
$\Lambda\cap\mathcal{Z} = \emptyset$. If $H$ is orthogonal to 
$\{g_\lambda\}_{\lambda\in\Lambda}$, we define
$$
Q(z) = \int_{\mathbb{R}} \frac{G(z)S(t) - G(t)S(z)}{t-z} \overline{H(t)} dt.
$$
As before, $Q$ vanishes on $\Lambda$. Using standard estimates for the 
functions in the Paley--Wiener space (e.g., the fact that $f$ and $f'$
are bounded on $\mathbb{R}$ for $f\in PW_\pi$), we conclude that 
$Q\in PW_\pi$ and, hence, $Q\equiv 0$.
Thus, $H$ is orthogonal to the system 
$\big\{\frac{S(z)}{z-\mu}\big\}_{\mu\in\mathcal{Z}}$ 
and, hence, is zero. }
\end{remark}

\end{document}